\documentclass[12pt]{article}
\usepackage{amsmath,amssymb,amsthm}
\usepackage{fullpage}

\newtheorem{thm}{Theorem}

\makeatletter
\def\section{\@startsection{section}{1}{\z@}{-3.5ex plus -1ex minus
			  -.2ex}{2.3ex plus .2ex}{\large\bf}}
\def\subsection{\@startsection{subsection}{2}{\z@}{-3.25ex plus -1ex
			  minus -.2ex}{1.5ex plus .2ex}{\normalsize\bf}}
\makeatother

\begin{document}

\title{Quotient polynomials with positive coefficients}
\author{Mark B. Villarino\\
Escuela de Matem\'atica, Universidad de Costa Rica,\\
11501 San Jos\'e, Costa Rica}
\date{\today}

\maketitle

\begin{abstract}
We give an optimal necessary and sufficient condition for the quotient
polynomial and remainder in the division algorithm to have positive
coefficients.
\end{abstract}


\section{Introduction} 

In a recent paper \cite{Vil2} Michael Hirschhorn and the present
author had to prove that the polynomial
\begin{equation}
\label{nasty-poly}
f(n) := 2842 n^5 - 7821 n^4 - 16884 n^3 + 10428 n^2 + 5082 n - 2607
\end{equation}
is positive for all integers $n \geq 4$. We did so by writing down the
identity
\begin{equation}
\label{nice-poly}
f(n) \equiv (2842 n^4 + 6389 n^3 + 15061 n^2 + 85733 n + 433747)(n - 5)
+ 2166128.
\end{equation}
Since \emph{all the coefficients of the quotient polynomial are
positive as well as the remainder}, this shows by inspection that
$f(n) > 0$ for $n \geq 5$ and computing $f(4) = 12025$ completes the
proof. This example illustrates a powerful method for proving that
polynomials have strictly positive values for all integers exceeding a
given integer. We learned this method from a paper by
\textsc{Chen}~\cite{CQ}, who continues to exploit the method to this
day (see~\cite{chen2} and the references cited there), and we have
subsequently used it in our own researches~\cite{Vil1,Vil2}.

We also conjectured in \cite{Vil2} that the theoretical basis of the
method is a true theorem and and our paper is dedicated to proving
this conjecture.

\begin{thm}[Conjecture] 
\label{ch}
Let $f(x)$ be a polynomial with real coefficients whose principal
coefficient is positive and with at least one positive root $x = a$.
Then there exists an $x = b \geq a$ such that the remainder resulting
from the division algorithm applied to the quotient
$\frac{f(x)}{x - b}$ is nonnegative and the nonzero coefficients of
the quotient polynomial are all positive numbers.
\end{thm}

We note that the example $f(x) := (x - 1)(x - 2)(x - 3)(x - 4)$ shows
that we \emph{cannot} take $b$ to be the first integer greater than
the largest positive root, $x = 4$, since the quotient by $x - 5$
gives $x^3 - 5x + 10$ which has a negative coefficient. We will see
below that $b$ \emph{must be greater than or equal to the largest of
all the positive roots of all the coefficient polynomials}
(see~\eqref{coeff} below) for the quotient polynomial and the
remainder polynomial. Indeed, in our example, the first $b$ that works
is $b = 10$.


\section{Laguerre's test} 

The location of the roots of a polynomial is a classical subject which
continues to be in the forefront of modern mathematical
investigations. The advent of computers and symbolic manipulation
programs have renewed and stimulated interest in this fascinating
subject.

In particular, recent interest in an upper bound to the largest
positive root of a polynomial with real coefficients has produced
important new contributions~\cite{Vig}.

In 1880, the distinguished French mathematician \textsc{E. N.
Laguerre} published a paper~\cite{Lgr} containing a test for an upper
bound for the largest positive root. This test has since become
standard material in textbooks on the theory of equations~\cite{U}.

\emph{This test also implicitly contains the theoretical basis of
Chen's method.}

Let
\begin{equation}
\label{f}
f(x) := a_0 x^n + a_1 x^{n-1} +\cdots+ a_n = 0
\end{equation}
be an equation with real coefficients of which the leading coefficient
$a_0 > 0$. Let $b$ be any positive number. Then the division algorithm
identity gives
\begin{equation*}
f(x) \equiv f_n(x)
\equiv (x - b) \{f_0(b) x^{n-1} + f_1(b) x^{n-2} +\cdots+ f_{n-1}(b)\}
+ f_n(b),
\end{equation*}
where the coefficient polynomials $f_k(x)$, $k = 0,1,\dots, n - 1$,
and the remainder polynomial $f_n(x) \equiv f(x)$ (which in fact is the
original polynomial) are defined by
\begin{gather}
f_0(x) := a_0, \qquad 
f_1(x) := x f_0(x) + a_1, \quad \cdots
\nonumber \\
f_{n-1}(x) := x f_{n-2}(x) + a_{n-1}, \qquad
f_n(x) := x f_{n-1}(x) + a_n,
\label{coeff}
\end{gather}
and where for any $k = 0,1,\dots,n$, the following identity is valid:
\begin{equation}
\label{ff}
f_k(x) 
\equiv (x - b) \{f_0(b) x^{k-1} + f_1(b) x^{k-2} +\cdots+ f_{k-1}(b)\}
+ f_k(b).
\end{equation}

Laguerre proved the following two theorems.

\begin{thm} 
\label{lag1}
If, for some positive number $b$, the numbers
\begin{equation*}
f_1(b), f_2(b),\dots, f_{n-1}(b)
\end{equation*}
are nonnegative and
\begin{equation*}
f_n(b) > 0,
\end{equation*}
then $x = b$ is an upper bound for the largest positive root of
$f(x) = 0$.
\end{thm}

\begin{proof}
By hypothesis and from the identity~\eqref{f} it follows that
\begin{equation*}
x \geq b  \implies  f_n(x) \equiv f(x) > 0,
\end{equation*}
so there cannot be any real root of the equation $f(x) = 0$ surpassing
$x = b$ or even equal to $x = b$.
\end{proof}

\begin{thm} 
\label{lag2}
If for some $b > 0$  and some $k \leq n$, the numbers
\begin{equation*}
f_1(b), f_2(b), \dots, f_k(b)
\end{equation*}
are nonnegative, then for $b' > b$ the numbers
\begin{equation*}
f_1(b'), f_2(b'), \dots, f_k(b')
\end{equation*}
are also nonnegative.
\end{thm}

\begin{proof}
Take $x = b'$ in the identity~\eqref{ff}. Then
\begin{equation}
\label{b'}
f_i(b') \equiv (b' - b) \{f_0(b) b'^{i-1} + f_1(b) b'^{i-2}
+\cdots+ f_{i-1}(b)\} + f_i(b)
\end{equation}
is a positive number for $i = 1,\dots,k$, since $b' > b$ and
$f_0(b) = a_0 > 0$.
\end{proof}


\section{Lagrange's lemma} 

In 1769, Lagrange published a memoir on the numerical solution of
algebraic equations~\cite{L} and stated the following result (which he
ascribes to MacLaurin).

\begin{thm}[Lagrange's lemma] 
Suppose $a_k$ is the first of the negative coefficients of the
polynomial $f(x)$ in~\eqref{f}. Then
\begin{equation}
\label{ll}
x > 1 + \sqrt[k]{\frac{B}{a_0}}  \implies  f(x) > 0,
\end{equation}
where $B$ is the largest absolute value of the negative coefficients
of the polynomial $f(x)$. Moreover, $1 + \sqrt[k]{\frac{B}{a_0}}$ is
greater than the largest positive root of~$f(x)$.
\end{thm}

\begin{proof}
We assume $x > 1$. If in $f(x)$ each of the nonnegative coefficients
$a_1,a_2,\dots,a_{k-1}$ be replaced by zero, and each of the remaining
coefficients $a_k,a_{k+1},\dots,a_n$ be replaced by the negative
number $-B$, we obtain
\begin{align*}
f(x) &\geq a_0 x^n - B(x^{n-k} + x^{n-k-1} +\cdots+ 1)
\\
&= a_0 x^n - B\,\frac{x^{n-k+1} - 1}{x - 1}
\\
&> a_0 x^n - \frac{B}{x-1}\, x^{n-k+1}
\\
&= \frac{x^{n-k+1}}{x - 1}\, \{a_0 x^{k-1}(x - 1) - B\}
\\
&> \frac{x^{n-k+1}}{x - 1}\, \{a_0(x - 1)^k - B\}
\\
&> 0,
\end{align*}
where the first two strict inequalities are due to the assumption
$x > 1$ and the last is due to the assumption
$$
x > 1 + \sqrt[k]{\frac{B}{a_0}} \,.
$$
Finally, the proof shows that the inequality $f(x) > 0$ is valid for
all $x > 1 + \sqrt[k]{\frac{B}{a_0}}$ which proves the last statement.
\end{proof}


\section{Proof of the conjecture} 

Now we give the proof of Theorem~\ref{ch} which is the theoretical
basis of Chen's method of proving that polynomials have strictly
positive values for all integers beyond an initial value. We will
prove the following \emph{optimal} theorem, which evidently suffices
to prove our conjecture.

\begin{thm}[Existence proof] 
Suppose that $x = a$ is the largest positive root of the polynomial
equation~\eqref{f}. Let $x = b$ be the largest positive root of all
the polynomials $f_1(x),\dots,f_n(x)$.

Then the numbers
\begin{equation*}
f_1(b), f_2(b), \dots, f_{n-1}(b), f_n(b)
\end{equation*}
are nonnegative and at least one is strictly positive. Moreover, no
value of $x$ smaller than $x = b$ will work, while every value greater
than or equal to $b$ will work.
\end{thm}

\begin{proof}
Since $f_n(x) \equiv f(x)$ has a positive root by hypothesis, the set
of positive roots of the coefficient polynomials $f_k(x)$,
$k = 1,\dots,n$, is not empty. Moreover, since $f_k(x)$ has
degree~$k$, the total number of positive roots cannot exceed
$1 + 2 +\cdots+ n = \frac{n(n + 1)}{2}\,$, which is a finite set of
real numbers. Therefore it has a maximal member, say $x = b$, which is
the largest positive root of $f_p(x)$, say.

Now we claim that \emph{every number $f_k(b)$, $k = 1,\dots,n$, is
nonnegative and that at least one is strictly positive}. If all were
equal to zero, then $f(x) = a_0 x^n$ which has no positive root, and
this contradicts the assumption $x = a > 0$ is a positive root of
$f(x)$. So, at least one $f_r(b) \neq 0$, with $r \neq p$. We claim
then that necessarily $f_r(b) > 0$ since if it were \emph{negative},
by Lagrange's lemma applied to $f_r(x)$, we would have to conclude
that
$$
b < 1 + \sqrt[k]{\frac{B}{a_0}}, \quad f_r(b) < 0, \quad
f_r \biggl( 1 + \sqrt[k]{\frac{B}{a_0}} \biggr) > 0,
$$
and by the intermediate value theorem $f_r(x) = 0$ at some value of
$x$ larger than $b$, which contradicts the maximality of~$b$.

Therefore, our maximal root $x = b$ will give us a quotient and
remainder whose nonzero coefficients are positive.

Moreover, \emph{no smaller value of~$b$ will work}. For, if division
of $f(x)$ by $x - b'$ where $0 < b' < b$ gives us a quotient and
remainder whose nonzero coefficients are positive, then the identity
\eqref{b'} with $i = p$ shows that $f_p(b') < 0$ since $b' - b < 0$;
and this contradicts the assumption that all nonzero coefficients are
nonnegative, since, in particular, we must therefore have
$f_p(b') \geq 0$.

Finally, Theorem~\ref{lag2} shows that every value of $x$ which is
greater than $x = b$ will also work.
\end{proof}

Note that we have given \emph{a necessary and sufficient condition}
for the quotient $\dfrac{f(x)}{x - b}$ to have a quotient polynomial
and remainder with all nonzero coefficients positive.


\section{Laguerre's algorithm} 

Although the existence of such an $x = b$ has been proved, finding the
largest positive root of all the coefficient polynomials and the
original polynomial rigorously is problematic. Laguerre therefore
proposed an algorithm to find a value of~$b$ which may not be optimal,
but which works and is easy to find.

We give Laguerre's own algorithm for finding such a value of $x = b$.

We find the smallest positive integer that makes $f_1(b)$ positive or
zero. Such a $b$ is easy to find since $f_1(x)$ is a linear polynomial
$f_1(x) = a_0 x + a_1$. If it turns out that all the numbers
$$
f_2(b), f_3(b), \dots, f_n(b)
$$
are nonnegative and $f_n(b) > 0$, then we can take this value of~$b$
as that whose existence is asserted in the theorem.

But, suppose that $f_{k+1}(b)$ is \emph{negative} while all the
preceding numbers
$$
f_1(b), f_2(b), \dots, f_k(b)
$$
are nonnegative. Then we repeat the process, trying
$x = b + 1, b + 2, \dots$ until a value $x = b_1$ is found that makes
$f_{k+1} \geq 0$. Since $f_{k+1} = x f_k + a_{k+1}$. such a~$b_1$ can
easily be found. By Theorem~\ref{lag2}, at the same time all the
numbers
$$
f_1(b_1), f_2(b_1), \dots, f_k(b_1)
$$
are nonnegative. Now, if
$$
f_1(b_1), f_2(b_1), \dots, f_{n-1}(b_1)
$$
are nonnegative and $f_n(b_1) > 0$, this value of $b$ is the value
whose existence is asserted in the theorem.

In the contrary case, we \emph{repeat} the process with a properly
chosen~$b_2$, etc. \emph{After at most $n$~trials, we will have found
a value of~$b$ as asserted in our conjecture}.

In practice, using a symbolic algebra program, one divides $f(x)$ by
$x - b$, $x - b_1$, etc., which makes it easy to quickly check the
positivity of the coefficients and the remainder.


\subsection*{Acknowledgment}
Support from the Vicerrector\'ia de Investigaci\'on of the University
of Costa Rica is acknowledged.


\end{document}